\documentclass{llncs}

\usepackage[utf8]{inputenc}

\usepackage[ruled,vlined]{algorithm2e}
\usepackage{lmodern}
\usepackage[utf8]{inputenc}
\usepackage[T1]{fontenc}
\usepackage{latexsym,exscale,amssymb,amsmath}
\usepackage{graphicx}
\usepackage[nointegrals]{wasysym}
\usepackage{eurosym}
\usepackage{hyperref}
\usepackage{caption}
\usepackage{svg}
\usepackage{blindtext}
\usepackage{bbm}
\usepackage{subcaption}
\usepackage{enumerate}

\DeclareMathOperator*{\argmin}{argmin}

\DeclareMathOperator{\dir}{dir}
\DeclareMathOperator{\prj}{P}

\newcommand{\bhv}{\mathbb{T}_N}
\newcommand{\ort}{\mathbb{S}}
\newcommand{\perpdir}{\left(\Sigma_{\ort} \bhv\right)^\perp}
\newcommand{\perpdirx}{\left(\Sigma_{x} \bhv\right)^\perp}
\newcommand{\Prb}{{\mathbb P}}
\newcommand{\Prbalt}{{\mathbb Q}}
\newcommand{\NN}{{\mathbb N}}
\newcommand{\RR}{{\mathbb R}}
\newcommand{\mean}{\mu}
\newcommand{\meanalt}{\nu}
\newcommand{\diff}{\mathrm{d}}
\newcommand{\wst}[2][1]{\mathcal{P}^#1\left(#2\right)}
\newcommand{\iid}{{\stackrel{i.i.d.}{\sim}}}

\begin{document}
\title{
Types of Stickiness in BHV Phylogenetic Tree Spaces and Their Degree%
\thanks{Supported by DFG GK 2088 and DFG HU 1575/7.}}
%
%
\author{Lars Lammers \inst{1} \and
Do Tran Van\inst{1}\and Tom M. W. Nye\inst{2} \and
Stephan F. Huckemann\inst{1}}
\authorrunning{L. Lammers et al.}
%
\institute{$^1$Felix-Bernstein-Institute for Mathematical Statistics 
in the Biosciences, University of G\"ottingen, 
Goldschmidtstrasse 7, 37077 G\"ottingen, Germany\\
$^2$School of Mathematics, Statistics and Physics, Newcastle University,\\ Newcastle upon Tyne, UK
\\
lars.lammers@uni-goettingen.de, 
do.tranvan@uni-goettingen.de, tom.nye@ncl.ac.uk
\\
stephan.huckemann@mathematik.uni-goettingen.de
}
\maketitle              
\begin{abstract}
It has been observed that the sample mean of certain probability distributions in Billera-Holmes-Vogtmann (BHV) phylogenetic spaces is confined to a lower-dimensional subspace for  large enough sample size. This non-standard behavior has been called stickiness and poses difficulties in statistical applications when comparing samples of sticky distributions. We extend previous results on stickiness to show the equivalence of this sampling behavior to topological conditions in the special case of BHV spaces. Furthermore, we propose to alleviate statistical comparision of sticky distributions by including the directional derivatives of the Fr\'echet function: the degree of stickiness.

\end{abstract}

\section{Introduction}

The Billera Holmes Vogtmann (BHV) spaces, first introduced in \cite{bhv}, are a class of metric spaces whose elements are trees describing potential evolutionary relations between species.  Allowing for statistical analysis of samples of entire phylogenies, they have gained considerable attraction in recent years. They are particularly attractive from a mathematical point of view as they were shown to be Hadamard spaces \cite[Lemma 4.1]{bhv}, i.e. complete metric spaces of global non-positive curvature. This results in many convexity properties of the metric guaranteeing, e.g. unique (up to reparametrization) geodesics between any two points and the existence and uniqueness of Fr\'echet means (see (\ref{eq:Frechetfcn}) below) for distributions with finite first moment \cite{sturm}. The Fr\'echet mean is a natural generalization of the expectation to any metric space $(M,d)$ as the minimizer of the expected squared distance of a probability distribution $\Prb \in \mathcal{P}^1(M)$. Here $\mathcal{P}(M)$ denotes the family of all Borel probability distributions on $M$ and
\begin{align*}
    \wst{M} = \Bigg \{ \Prb \in \mathcal{P}(M) \ \Big \vert \ \forall x \in M : 
        \int_M d(x,y) \ \diff \Prb(x) < \infty \Bigg \}.
\end{align*}
For a distribution $\Prb \in \wst{M}$, the Fr\'echet mean $b(\Prb)$ is then the set 
of minimizers of the \emph{Fr\'echet function}
\begin{align}\label{eq:Frechetfcn}
    F_\Prb(x) = \frac{1}{2} \int_M \left( d^2(x,y) - d^2(z,y)\right) \diff \Prb(y) 
        \quad x \in M,
\end{align}
for arbitrary $z \in M$. In Hadamard spaces, while the Fr\'echet function is strictly convex, 
in certain spaces, sampling from some distributions leads to degenerate behavior of the sample Fr\'echet mean, where, after a finite random sample size, it is restricted to a lower dimensional subsets of the space. This phenomenon has been called stickiness and was studied for various spaces, including BHV spaces, see \cite{bhvclt2,barden18,openbook,kale}. This absence of asymptotic residual variance or its reduction  incapacitates or aggravates standard statistical methodology. 

In \cite{kale}, a topological notion of stickiness was proposed: 
given a certain topology on a set of probability spaces, a distribution \emph{sticks} to $S \subset M$ if all distributions in a sufficiently small neighborhood have their Fr\'echet means  in $S$. There, it was also shown for the so-called \emph{kale} that \emph{sample stickiness} is equivalent to \emph{topological stickiness}, induced by equipping $\wst{M}$ with the Wasserstein distance
\begin{align*}
    W_1(\Prb, \Prbalt) = \argmin_{\pi \in \Pi(\Prb, \Prbalt)}   
        \int_{M \times M} d(x,y) \diff \pi(x,y),
\end{align*}
where $\Pi(\Prb, \Prbalt)$ denotes the set of all couplings of $\Prb, \Prbalt 
\in \wst{M}$.

In this paper, we provide this equivalence of both notions of stickiness for strata of BHV spaces with positive codimension by using directional derivatives of the Fr\'echet function. Furthermore, we propose using these directional derivatives as a tool to discriminate between sticky distributions whose means are indistinguishable.

\section{The Billera-Holmes-Vogtmann Phylogenetic Tree Space}

For $N \in \mathbb{N}$ and $N \geq 3$, the BHV tree space $\bhv$ represents rooted trees with $N$ labelled leaves via positive lengths of \emph{interior edges}. Here, an interior edge is a split of the set of leaves and the root with at least two elements in both parts. The set of splits of a tree determines its \emph{topology}. Whenever new internal nodes appear or existing ones coalesce, the topology changes. 
\begin{definition}
    Trees with common topology form a  \emph{stratum $\ort \subset \bhv$}. We say the stratum is of codimension $l \geq 0$
    if the topology features $N-2-l$ splits.
\end{definition}
The highest possible stratum dimension is $N-2$, which happen in the topology of a binary, i.e. of a fully resolved tree.

 Taking the Euclidean geometry within closed strata and gluing them together at their boundaries, \cite{bhv} arrive at the  separable Hadamard space $(\bhv,d))$.  Thus, geodesics between two trees in a BHV tree space correspond to changing the length of splits present in both trees and the addition and removal of the other splits. Geodesics between two trees can be computed in polynomial time 
\cite{gtp}.
 
 For $x\in \bhv$ let $B_\epsilon(x)$ be the open ball of radius $\epsilon>0$ in $\bhv$. Then, in direct consequence of the construction, we have the following.
 
\begin{lemma}\label{lem:ball_in_bhv}
    Let $\ort \subset \bhv$ be any stratum $l \geq 1$. Then, for any $x \in 
    \ort$, there is $\epsilon > 0$ such that $\overline{B_\epsilon(x)} 
    \cap \ort$ is closed in $\ort$. Furthermore, the topology of any $y \in 
    \overline{B_\epsilon(x)}$ features all splits present in $x$.
\end{lemma}

\vspace*{-1cm}
\begin{figure}
    \centering
    \includegraphics[width=10cm, height=7.5cm]{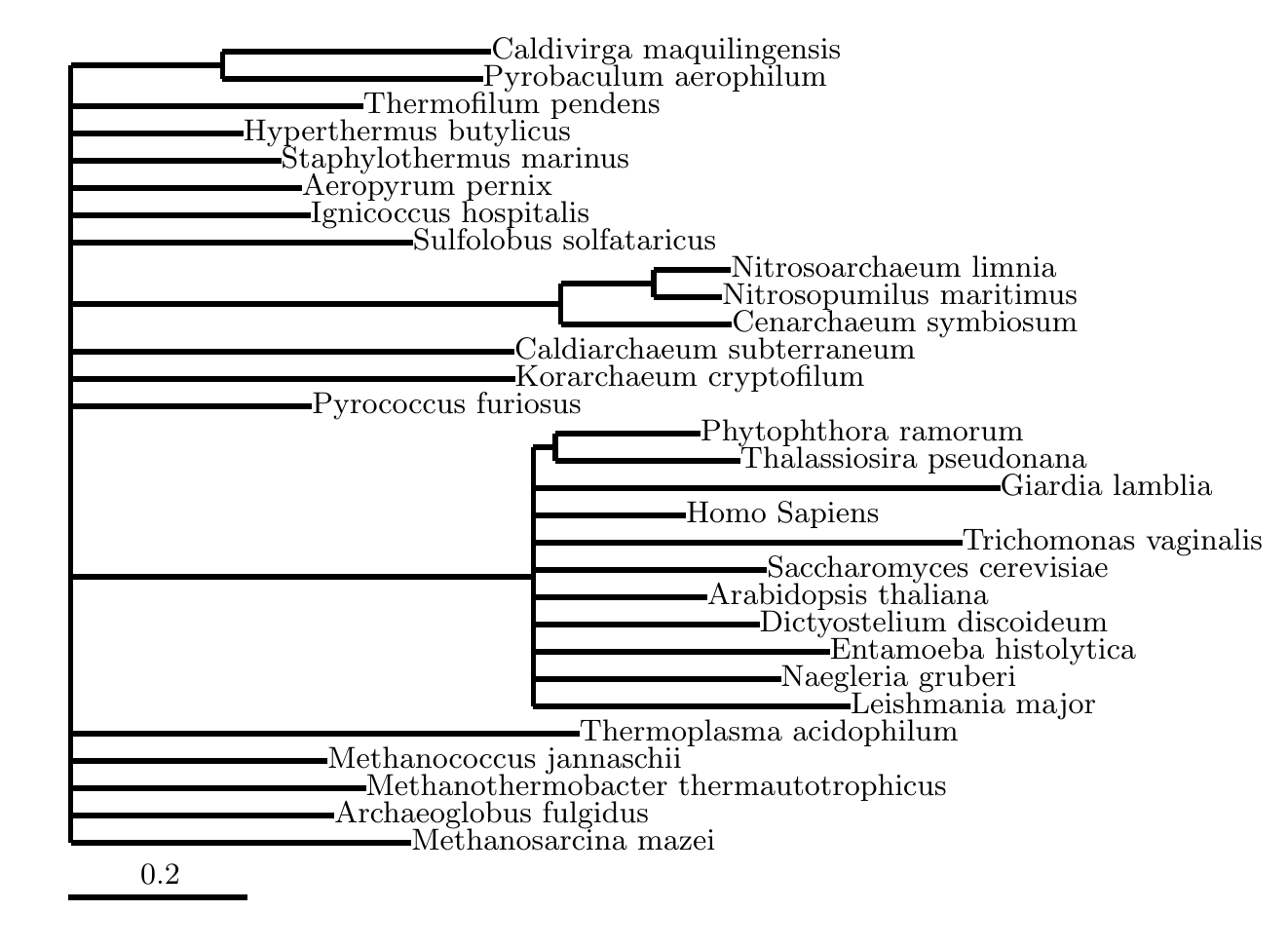}
    \caption{\it Placing eukaryotes (homo sapiens) within the archaea: The Fr\'echet mean of a data set from \cite{lgc20} is  a highly unresolved phylogenetic tree. Lengths of horizontal lines correspond to evolutionary distance, vertical lines to common nodes, the leftmost vertical stands for the common root.
    }
    \label{fig:mean}
\end{figure}
 
\vspace*{-1cm}
\section{Properties as Hadamard Spaces}

For an extensive overview, we refer to \cite{bridson,burago,sturm}. In a Hadamard space $(M,d)$, a function $f: M \to \mathbb{R}$ is (strictly) convex if all compositions with geodesics are (strictly) convex. Convex sets are sets containing all geodesic segments between two points of the set. An example of a strictly convex function is the Fr\'echet function. 
Besides this convexity 
we will also require the following results. 

\begin{theorem}[Theorem 2.1.12 in \cite{bacak}]
\label{theorem:metric_proj}
Let $(M,d)$ be a Hadamard space and let $S \subset M$ be a closed convex 
set. Then the following statements hold true for the \emph{metric projection} 
$\prj_S: M \to S, x \mapsto \argmin_{y \in S} d(x,y)$.
    \begin{enumerate}
        \item The map $x \mapsto \prj_S(x)$ is single-valued for all $x \in M$.
        \item It holds for any $x,y \in M$ that $d \left( \prj_S(y), \prj_S(x) 
        \right) \leq d(x, y)$.
    \end{enumerate}
\end{theorem}

\begin{remark}
As the strata of BHV Spaces are locally Euclidean, the metric projection of a tree to
an \emph{adjacent} stratum of positive codimension corresponds to simply removing
the redundant splits and keeping the splits featured in the topology of the 
stratum and their  respective lengths.
\end{remark}
    
\begin{theorem}[Theorem 6.3 in \cite{sturm}]
    \label{theorem:contraction}
    Let $(M,d)$ be a separable Hadamard space. Then $       d(b(\Prb), b(\Prbalt)) \leq W_1(\Prb, \Prbalt)$ for any two $\Prb, \Prbalt \in \wst{M}$.

\end{theorem}
    
\section{The Space of Directions}

In a Hadamard space $(M,d)$, it is possible to compute an angle between two 
(non-constant) geodesics $\gamma, \gamma^\prime$ starting at the same point 
$x \in M$. This angle is called the \emph{Alexandrov angle} $\angle_x$ and 
can be computed as follows \cite[Chapter 3]{burago}
\begin{align*}
    \angle_x(\gamma, \gamma^\prime) = \lim_{t, t^\prime \searrow 0} \arccos
        \left( \frac{t^2+ {t'}^2 -d(\gamma(t), \gamma^\prime(t^\prime))}
        {2 t \cdot t^\prime}\right).
\end{align*}
Two geodesics have an equivalent direction at $x$ if the Alexandrov angle 
between them is 0. The set of these equivalence classes is called the \emph{space
of directions at $x$} and is denoted by $\Sigma_x M$. Equipped with the Alexandrov
angle, the space of directions becomes a spherical metric space itself.  For an overview, see 
e.g. Chapter 9 in \cite{burago}. For points  $x,y, z \in M$, we write $\dir_x(y)$ 
for the direction of the (unit speed) geodesic from $x$ to $y$ and we write 
$\angle_x(y,z) = \angle_x(\dir_x(y), \dir_x(z))$.\\

Let $\ort \subset \bhv$ be a stratum with positive codimension $l \geq 1$ and
set
\begin{align*}
    \perpdirx &= \{ \dir_x(z) \vert z \neq x, \ z \in \prj_{\overline{\ort}}^{-1}(\{x\}) \},\\
    \left(\Sigma_x \bhv \right)^\parallel &= \{ \dir_x(z) \vert z \neq x, \ z \in \ort \}.
\end{align*}
The following lemma is concerned with the structure of the space of directions 
in BHV tree spaces. It is inspired by the work of the tangent cone of orthant
spaces in \cite{barden18}. 
For two metric spaces 
$(M_1, d_1), (M_2, d_2)$, recall their \emph{spherical join}
$$M_1  \ast M_2 = \left[0,\frac{\pi}{2}\right]\times M_1\times M_2/\sim \cong \left\{(\cos \theta\, p_1,\sin \theta\, p_2): 0\leq \theta \leq \frac{\pi}{2}, p_i\in M_i,i=1,2\right\}$$
with the metric
$$ d\big((\theta,p_1,p_2),(\theta,p_1,p_2)\big) = \arccos\left(\cos\theta\cos\theta' d_1(p_1,p'_1) + \sin\theta \sin\theta' d_2(p_2,p'_2)\right)\,.$$
In particular, this turns $M_1*M_2$ into a sphere of dimension $n_1+n_2+1$ if $M_i$ is a sphere of dimension $n_i$, $i=1,2$.

\begin{lemma}
    \label{lemma:tancone}
    Let $\ort \subset \bhv$ be a stratum with positive codimension $l \geq 1$ and $x \in \ort$. Then its space of directions can be given the structure of a 
    spherical join
    \begin{align*}
        \Sigma_x \bhv \cong \left(\Sigma_x \ort\right)^\parallel \ast \perpdir, \quad \perpdir\cong \perpdirx
    \end{align*}
\end{lemma}

\begin{proof}
    For sufficiently small $\epsilon > 0$, with $B_\epsilon(x)$ from Lemma \ref{lem:ball_in_bhv} for any geodesic starting at $x$ we have a one-to-one correspondence between its 
    direction $\sigma \in \Sigma_x \bhv$ and a point $y_\sigma \in \partial B_\epsilon(x)$ with $d(x,y_\sigma) = 
    \epsilon$. This gives rise to the angular part of the join 
    \begin{align*}
        \theta: \Sigma_x \bhv \to [0, \pi/2], \sigma \mapsto \arcsin\left(\frac{d(y_\sigma, \ort)}{\epsilon}\right)\,.
    \end{align*}
    Furthermore, as remarked before Lemma \ref{lem:ball_in_bhv}, the topology of $y$ features
    all splits of $x$ and at most $l$ additional splits. With the map $y\mapsto x_y : 
    \partial\overline{B_\epsilon(x)}\to \overline{B_\epsilon(x)}$,  adding to $x$ all splits of $y$, not present in $x$, with their lengths from $y$, set
    $$y^\perp := x_y,\quad y^\parallel :=  P_{\overline{\ort}}(x_y)\,,$$
    where we identify the directions of $x_y$ and $x'_y$ at $x\in\ort$ and $x'\in \ort$, respectively, if their split lengths after removing those of $x$ and $x'$, respectively, agree.
    In conjunction with
    \begin{align*}
        \phi^\perp(\sigma) := \dir_x(y^\perp_\sigma),
            \quad \phi^\parallel(\sigma) := \dir_x(y^\parallel_\sigma),
    \end{align*}
     thus obtain, with the second factor independent of the base point,   \begin{align*}
        \Phi : &\Sigma_x \bhv \to \left(\Sigma_x 
            \ort\right)^\parallel \ast \perpdirx,\quad
        \sigma \mapsto \left(\theta(\sigma), \phi^\parallel(\sigma), 
            \phi^\perp(\sigma)\right).
    \end{align*}
    Straightforward computation verifies that  $\Phi$ is a bijection.  
    
    It remains to show that $\Phi$ is isometric. For notational simplicity, 
    suppose $\epsilon = 1$.  Let $\sigma_1, \sigma_2 \in \Sigma_x \bhv$ with 
    $y_i := y_{\sigma_i}, i=1,2$ and
        $r_i^\perp := d(y_i, \ort)$, $r_i^\parallel := \sqrt{1 - 
            (r_i^\perp)^2}$. Exploiting $d^2(y_1, y_2) = 
    d^2(y_1^\parallel, y_2^\parallel) + d^2(y_1^\perp,y_2^\perp)$ by definition of the geodesic distance and that $\bhv$ is Euclidean in each stratum, we have indeed,
    {\footnotesize
    \begin{align*}
        &\cos(\angle_x(y_1, y_2)) = \frac{1 + 1 - d^2(y_1,y_2)}{2}\\
        &= \frac{(r_1^\perp)^2 + (r_2^\perp)^2 - d^2(y_1^\perp,y_2^\perp)}{2}
            +  \frac{(r_1^\parallel)^2 + (r_2^\parallel)^2 - d^2(y_1^\parallel,
            y_2^\parallel)}{2}\\
        &= r_1^\parallel \cdot r_2^\parallel \cdot (\cos(\angle_x(y_1^\parallel, 
            y_2^\parallel)) + r_1^\perp \cdot r_2^\perp \cdot 
            (\cos(\angle_x(y_1^\perp, y_2^\perp))\\
        &= \cos(\theta(\sigma_1)) \cos(\theta(\sigma_2)) 
            \cos\left(\angle_x(\phi^\parallel(\sigma_1), \phi^\parallel(\sigma_2))\right)\\
        &\quad + \sin(\theta(\sigma_1)) \sin(\theta(\sigma_2)) 
            \cos\Big(\angle_x(\phi^\perp(\sigma_1), \phi^\perp(\sigma_2))\Big).
    \end{align*}
    }
\end{proof}
    In light of this fact, we shall henceforth abuse notation and will identify $\perpdirx \cong\perpdir$  and any
$\sigma \in \perpdir$ with natural embedding into $\Sigma_x M$ for all 
 $x \in \ort$.

\section{Equivalent of Notions of Stickiness in BHV Spaces}


For a sequence $(X_i)_{i \in \mathbb{N}}$ of i.i.d. random variables following a distribution $\Prb$, let $\Prb_n = \frac{1}{n} \sum_{i=1}^n X_i$ denote the empirical measure.

\begin{definition}[Three Notions of Stickiness]
\label{def:sticky}
Let $\Prb \in \wst{\bhv}$ be a probability distribution in a BHV space and $\ort \subset \bhv$ be a stratum with codimension $l \geq 1$. Then on $\ort$, $\Prb$ is called
\begin{description} 
\item[Wasserstein sticky] if there is  $\epsilon > 0$ such that $b(\Prbalt) \in \ort$ for all $\Prbalt \in \wst{\bhv}$ with $W_1(\Prb,\Prbalt) 
< \epsilon$,
\item[perturbation sticky]
if for any $\Prbalt \in \wst{\bhv}$ there is $t_\Prbalt > 0$ such that  $b((1-t)\Prb + t \Prbalt) \in \ort$ for all $0 < t <  t_\Prbalt$,
\item[sample sticky] if for any sequence of random variables $(X_i)_{i \in  \mathbb{N}}\iid\Prb$ 
there is a random $N\in \NN$ such that $b(\Prb_n) \in \ort$ for all $n \geq N$ a.s.
\end{description}
\end{definition}



\begin{theorem}[\cite{lammers2023types}]
    \label{lemma:lip_angle}
    \label{lemma:lip_wst}
    \label{thm:dir-der-formula}
Let $(M,d)$ be a Hadamard space, $x \in M$, and $\gamma: [0, L] \to M$ be a unit speed geodesic with direction $\sigma$ at $\gamma(0) = x$. Then for any $\Prb \in \wst{M}$, the directional derivative
of the Fr\'echet function $\nabla_\sigma F_\Prb (x) = 
\frac{\text{d}}{\text{d}t} F_\Prb(\gamma(t))\lvert_{t=0}$ exists, is 
well-defined and
\begin{align*}
    \nabla_\sigma F_\Prb (x) = - \int_M \cos(\angle_x(\sigma, \dir_x(z)))
        \cdot d(x,z) \ \diff\Prb(z).
\end{align*}
In particular, it is
\begin{enumerate}
    \item Lipschitz continuous as a map $\Sigma_x \bhv \to \mathbb{R},
        \sigma \mapsto \nabla_\sigma F_\Prb(x)$, and
    \item 1-Lipschitz continuous as a map $\wst{M} \to \mathbb{R},
        \Prb \mapsto \nabla_\sigma F_\Prb(x)$.
\end{enumerate}
\end{theorem}

The following 
result follows directly from
Lemma \ref{lemma:tancone} and Theorem \ref{lemma:lip_angle}.

\begin{corollary}
\label{lemma:const_der}
    Let $\ort \subset \bhv$ be a stratum of positive codimension $l \geq 1$, 
    let $\Prb  \in \wst{\bhv}$ and identify $\sigma \in \perpdir$ for all 
    $x \in \ort$ across all spaces of directions in $\ort$. Then
     $x \mapsto \nabla_\sigma F_\Prb (x), \ort \to \RR$ is constant.
\end{corollary}


\textbf{Assumption 1:} 
For $X \sim \Prb \in \wst{\bhv}$ with $b(\Prb) = \mean 
\in \ort$ for a stratum $\ort \subset \bhv$ assume that
\begin{align*}
    \Prb\{\phi_{\sigma}(X) = 0\} < 1 \mbox{ for all }\sigma \in \perpdir,
\end{align*}
where for $z\in \bhv$, 
$    \phi_\sigma(z) = -d(z,\mean) \cos(\angle_\mean(\sigma,\dir_\mean(z)))\,.$

\begin{theorem}
\label{theorem:main}
Let $N \geq 4$ and consider a stratum $\ort \subset \bhv$ with positive
codimension $l \geq 1$. Then the following statements are equivalent for a
probability distribution $\Prb \in \wst{\bhv}$ with $\mean = b(\Prb) \in \ort$.
\begin{enumerate}
    \item $\Prb$ is Wasserstein sticky on $\ort$.
    \item $\Prb$ is perturbation sticky on $\ort$.
    \item $\Prb$ is sample sticky on $\ort$ and fulfills Assumption 1.
    \item For any direction $\sigma \in \perpdir$, we have that 
    $\nabla_\sigma  F_\Prb(\mean) > 0$.
\end{enumerate}
\end{theorem}

\begin{proof}
"$1 \implies 2$" follows at once from 
$W_1(\Prb, (1-t)\Prb + t \Prbalt) = t W_1(\Prb,
\Prbalt)$ as verified by direct computation. 

"$2 \implies 4$":
Let $\sigma \in \perpdir$, $y \in \bhv$ such that
$\prj_\ort(y)= \mean$ and $\sigma = \dir_\mean(y)$, and let $\Prbalt_t = (1-t)\Prb + t \delta_y$, for $0 \leq t \leq 1$. By 
hypothesis, we find $0 < t_y < 1$ such that $b(\Prbalt_t) \in \ort$ for all $t \leq t_y$.
Since $\mean$ is the Fr\'echet mean of $\Prb$ and 
$\prj_\ort(y) = \mu$, we have 
for any $x \in \ort \setminus \{\mean\}$ that $F_{\Prbalt_t}(x)> F_{\Prbalt_t}(\mean)$.
Thus, $\mean$ is be the Fr\'echet mean of $\Prbalt_t$ for all $t \leq t_y$, 
and hence for all $t \leq t_y$,
\begin{align*}
    0 \leq \nabla_\sigma F_{\Prbalt_t} (\mean) 
        = (1 - t) \cdot \nabla_\sigma F_{\Prb} (\mean) - t \cdot d(\mean, y),
\end{align*}
whence $\nabla_\sigma F_{\Prb} (\mean) \geq \frac{t_y}{1-t_y} d(\mean, y) > 0$.

"$4 \implies 1$": Take arbitrary $\eta >0$ such that $\overline{B_\eta(\mean)} 
\cap \ort$ is closed in $\ort$. As $\perpdir$ is compact and
the directional derivatives are continuous in directions
(Theorem \ref{lemma:lip_angle}), there is a lower bound
\begin{align*}
    0 < \zeta = \min_{\sigma \in \perpdir} \nabla_\sigma F_\Prb 
        \left( \mean \right).
\end{align*}
Then, for any $\Prbalt \in \wst{\bhv}$ with $W_1(\Prb, \Prbalt) < \epsilon: = \min \{ \zeta, \eta\}$, due to Theorem \ref{theorem:contraction}), it  follows that $d(\mean, b(\Prbalt)) < \epsilon \leq \eta.$ By Lemma \ref{lem:ball_in_bhv}, the topology of $\meanalt = b(\Prbalt)$ must feature all splits in the topology of $\mean$. 

It is left to show that $b(\Prbalt) \notin \ort$ cannot be. Otherwise, with $y = 
\prj_{\overline{\ort}}(\meanalt)$, by Theorem \ref{theorem:metric_proj}, we have $d(\mean, y) \leq d(\mean, \meanalt) < \eta$ and hence, $y \in 
\ort$. Furthermore,  $\sigma=\dir_y(\meanalt) \in \perpdir$ and thus by Theorem \ref{lemma:const_der} and Corollary \ref{lemma:const_der},
\begin{align*}
    \lvert \nabla_\sigma F_\Prb(\mean) - \nabla_\sigma F_\Prbalt(y)\rvert \leq
        W_1(\Prb, \Prbalt) < \epsilon \leq \zeta.
\end{align*}
Hence,  $\nabla_\sigma F_\Prbalt(y) > 0$, which implies, following $\sigma$, by strict convexity of the Fr\'echet function, that $F_\Prbalt(y) < F_\Prbalt(\meanalt)$,  
so that $\meanalt$ is not the Fr\'echet mean of $\Prbalt$ .

"$3 \implies 4$":
Since $b(\Prb) = \mean$, we have $\nabla_\sigma F_\Prb(\mean) 
\geq 0$ for all $\sigma \in \perpdir$. We now show that $\nabla_{\sigma^\prime} F_\Prb(\mean) 
= 0$ for some 
$\sigma^\prime \in \perpdir$ yields a contradiction. 
Indeed, then for $(X_i)_{i \in \mathbb{N}} \iid \Prb$, there is $\Prb$-a.s. a random number $N\in \mathbb{N}$ such that $b(\Prb_n) \in \ort$ for all $n \geq N$. Due to Theorem \ref{thm:dir-der-formula}, we also have $\Prb$-a.s.,
\begin{align}
    \label{ineq:nonnegder}
    \nabla_{\sigma^\prime} F_{\Prb_n}(\mean) \geq 0 \quad \forall n \geq N.
\end{align}
Using the notation of Assumption 1, 
consider $S_n = \sum_{i=1}^n \phi_{\sigma^\prime}
(X_i)$, so that $\nabla_{\sigma^\prime} F_{\Prb_n}(\mean) = S_n/n$.
Recalling that we assumed that $\nabla_{\sigma^\prime} F_\Prb(\mean) = \mathbb{E}(\phi(X_i)) = 0$, which implies
that the random walk $S_n$
is recursive (Theorem 5.4.8 in \cite{durrett}), and hence (Exercise 5.4.1 in \cite{durrett}) either
$$ \Prb\{ S_n=0: \mbox{ for all } n\in\NN\} =1\mbox{ or } \Prb\left\{ -\infty = \liminf_{n \in \mathbb{N}} S_n< 
    \limsup_{n \in \mathbb{N}}S_n = \infty\right\} =1\,.$$
The former violates Assumption 1, the latter contradicts (\ref{ineq:nonnegder}), however.

"$1 \mbox{ and }4 \implies 3$": By  \cite[Theorem 6.9]{villani08}, $\Prb_n$ converges against $\Prb$ in $W_1$, hence $\Prb$ is sample sticky. As the directional derivatives of the Fr\'echet function for any $\sigma \in \perpdir$ are non-zero by hypothesis, Assumption 1 holds.

\end{proof}

\section{Application: The Degrees of Stickiness}

\begin{definition}
    Let $\Prb \in \wst{\bhv}$ be a distribution that is Wasserstein sticky on a 
    stratum $\ort$ with positive codimension with $\mean = b(\Prb)$. Then, 
    we call $D_\sigma F_\Prb(\mu)$ the \emph{degree of stickiness of $\Prb$ in 
    direction $\sigma$}.
\end{definition}

We propose to use the degrees of stickiness as a way to discriminate between samples that are sticky on the same stratum. The following example illustrates such an application with two phylogenetic data sets $X,Y$ from \cite{lgc20} with empirical distributions $\Prb^X$ and $\Prb^Y$, where each consists of 63 phylogenetic trees that were inferred from the same genetic data using two different methods. The resulting two Fr\'echet mean trees $\mean_X, \mean_Y$ (after pruning very small splits) coincide in their topologies, as  displayed in Figure \ref{fig:mean}. We test the hypothesis
\begin{align*}
    \mathcal{H}_0 : D_\sigma F_{\Prb^X} (\mean_X) = D_\sigma F_{\Prb^Y} (\mean_Y) \quad 
        \forall \sigma \in \Sigma_{X,Y},
\end{align*}
where we choose $\Sigma_{X,Y}\subset \perpdir$ comprising only directions corresponding to a single split that is present in either $X$ or $Y$ and compatible with the topologies of $\mean_X$ and $\mean_Y$. As there is a natural pairing, we performed a pairwise t-test for each of the directions and applied a Holm-correction, leading to a p-value of 0.0227. This endorses the observation in \cite{lgc20}, that the two methods inferring phylogenetic trees differ significantly on this data set.

\bibliographystyle{splncs04}
\bibliography{references}
\end{document}